\documentclass[12pt,reqno]{amsart}
\usepackage{amssymb, amsfonts, amsbsy, latexsym, epsfig, color}

\newtheorem{Thm}{Theorem}[section]

\newtheorem{corollary}[Thm]{Corollary}
\newtheorem{lemma}[Thm]{Lemma}

\newtheorem{proposition}[Thm]{Proposition}
\newtheorem{definition}[Thm]{Definition}
\newtheorem{remark}[Thm]{Remark}

\newtheorem{example}[Thm]{Example}
\newtheorem{theorem}[Thm]{Theorem}

\newcommand{\bitem}{\begin{itemize}}
\newcommand{\eitem}{\end{itemize}}
\newcommand{\ip}[2]{\langle#1,#2\rangle}

\newcommand{\spann}{\mbox{\rm span}}

\def\ldots{\mathinner{\ldotp\ldotp\ldotp}}
\def\cdots{\mathinner{\cdotp\cdotp\cdotp}}

\def \tr{\text{tr }}

\def \cE{\mathcal{E}}
\def \cK{\mathcal{K}}
\def \cV{\mathcal{V}}
\def \cW{\mathcal{W}}

\def \spann{\text{span}}
\def \beq{\begin{equation}}
\def \eeq{\end{equation}}

\newcommand{\bR}{{\mathbb R}}
\newcommand{\NN}{{\mathbb N}}

\newcommand{\RR}{{\mathbb R}}

\def \0{\mathbf{0}}
\def \tr{\textnormal{tr}}

\def\cH{\mathcal{H}}

\newcommand{\ap}[1]{{\color{black}{#1}}}

\begin{document}

\title{Fusion Frames: Existence and Construction}

\author[R. Calderbank]{Robert Calderbank}
\address{Program in Applied and Computational Mathematics, Princeton University, Princeton, NJ 08544-1000, USA}
\email{calderbk@math.princeton.edu}

\author[P. G. Casazza]{Peter G. Casazza}
\address{Department of Mathematics, University of Missouri, Columbia, MO 65211-4100, USA}
\email{pete@math.missouri.edu}

\author[A. Heinecke]{Andreas Heinecke}
\address{Department of Mathematics, University of Missouri, Columbia, MO 65211-4100, USA}
\email{ah343@mizzou.edu}

\author[G. Kutyniok]{Gitta Kutyniok}
\address{Institute of Mathematics, University of Osnabr\"uck, 49069 Osnabr\"uck, Germany}
\email{kutyniok@math.uni-osnabrueck.de}

\author[A. Pezeshki]{Ali Pezeshki}
\address{Electrical and Computer Engineering, Colorado State University, Fort Collins, CO 80523-1373, USA}
\email{pezeshki@engr.colostate.edu}

\thanks{P.G.C. was supported by NSF DMS 0704216. G.K. would like to thank the
Department of Statistics at Stanford University and the Mathematics
Department at Yale University for their hospitality and support
during her visits. She was supported by Deutsche
Forschungsgemeinschaft (DFG) Heisenberg-Fellowship KU 1446/8-1. R.C.
and A.P. were supported in part by the NSF under Grant 0701226, by
the ONR under Grant N00173-06-1-G006, and by the AFOSR under Grant
FA9550-05-1-0443. The authors would like to thank the American
Institute of Mathematics in Palo Alto, CA, for sponsoring the
workshop on ``Frames for the finite world: Sampling, coding and
quantization'' in August 2008, which provided an opportunity for the
authors to complete a major part of this work.}

\begin{abstract}
Fusion frame \ap{theory is an \ap{emerging} mathematical theory that
provides a natural framework for performing} hierarchical data
processing. A {\em fusion frame} \ap{is} a frame-like collection of
subspaces in a Hilbert space, thereby generalizing the concept of
\ap{a frame for signal representation. In} this paper, we \ap{study}
the existence and construction of fusion frames. We first present a
complete characterization of \ap{a special class of fusion frames,
called \textit{Parseval} fusion frames. The value of Parseval fusion
frames is that the inverse fusion frame operator is equal to the
identity and therefore signal reconstruction can be performed
with minimal complexity.} We \ap{then} introduce
\ap{two general methods} -- the {\em spatial complement} and
the {\em Naimark complement} -- for constructing \ap{a new fusion
frame from a given fusion frame}. We \ap{then establish existence
conditions for fusion frames with desired properties. In particular, we address the
following question: Given} $M, N, m \in \NN$ and
$\{\lambda_j\}_{j=1}^M$\ap{, does} there \ap{exist} a fusion frame
in $\RR^M$ with $N$ subspaces of dimension $m$ for which
$\{\lambda_j\}_{j=1}^M$ are the eigenvalues of the associated fusion
frame operator\ap{?} We \ap{address} this problem by providing
an algorithm which computes such a fusion frame for almost \ap{any}
collection of \ap{parameters $M, N, m \in \NN$ and
$\{\lambda_j\}_{j=1}^M$.} Moreover, we \ap{show} how this procedure
can be applied, if subspaces \ap{are to} be added to a given fusion
frame to force it to become \ap{Parseval}.
\end{abstract}



\maketitle

\section{Introduction}

\subsection{\ap{Fusion Frames}}

Recent advances in hardware \ap{technology have} enabled the
\ap{economic production and deployment} of sensing and computing
networks consisting of a large number of low-cost components, which
through collaboration enable reliable and efficient operation.
\ap{Across different disciplines there is a fundamental shift from
centralized information processing to distributed or network-wide
information processing. Data communication is shifting from
point-to-point communication to packet transport over wide area
networks where network management is distributed and the reliability
of individual links is less critical. Radar imaging is moving away
from single platforms to multiple platforms that cooperate to
achieve better performance. Wireless sensor networks are emerging as
a new technology with the potential to enable cost-effective and
reliable surveillance.} These applications typically involve a large
number of data streams, which need to be integrated at a central
processor. Low communication bandwidth and limited
transmit/computing power at each single node in the network give
rise to the need for decentralized data analysis, where data
reduction/processing is performed in two steps: local processing at
neighboring nodes followed by the integration of locally
processed data streams at a central processor.

{\em Fusion frames} \ap{(or {\em frames of subspaces}) \cite{CKL08}
are a recent development} that \ap{provide} a natural mathematical
framework for \ap{two-stage (or, more generally, hierarchical) data
processing. The notion of a fusion frame} was introduced in
\cite{CKL08} with the main ideas already contained in \cite{CK04}. A
fusion frame \ap{is} a frame-like collection of subspaces in a
Hilbert space. In frame theory, a signal is represented by a
collection of {\em scalars}, which measure the amplitudes of the
projections of the signal onto the frame vectors, whereas in fusion
frame theory the signal is represented by the projections of the
signal onto the fusion frame subspaces. In a two-stage data
processing setup, these projections serve as locally processed data,
which can be combined to reconstruct the signal of interest.

Given a Hilbert space $\cH$ and a family of closed subspaces $\{\cW_i\}_{i \in I}$ with
associated positive weights $v_i$, $i \in I$, a {\em fusion frame} for $\cH$ is a collection
of weighted subspaces $\{(\mathcal{W}_i,v_i)\}_{i \in I}$ such that there exist
constants $0 < A \le B < \infty$ satisfying
\[
A\|f\|^2 \le \sum_{i \in I} v_i^2 \|P_i f\|^2 \le B\|f\|^2 \qquad \mbox{for any }f \in \cH,
\]
where $P_i$ is the orthogonal projection onto $\mathcal{W}_i$. The constants $A$ and $B$ are called
{\em fusion frame bounds}. A fusion frame is called {\em tight}, if $A$ and $B$ can be chosen
to be equal, and {\em Parseval} if $A=B=1$. If $v_i = 1$ for all $i \in I$, for the sake of brevity,
we sometimes write $\{\mathcal{W}_i\}_{i \in I}$ instead of $\{(\mathcal{W}_i,1)\}_{i \in I}$.

\ap{Any} signal $f \in \cH$ can be \ap{reconstructed \cite{CKL08}} from its fusion frame
measurements $\{v_i P_i f\}_{i \in I}$ by performing
\begin{equation} \label{eq:FFreconstr}
f = \sum_{i \in I} v_i S^{-1} (v_i P_i f),
\end{equation}
where $S = \sum_{i \in I} v_i P_i f$ is the {\em fusion frame operator} known to be positive and
self-adjoint.

\begin{remark}
\ap{If we wish to perform} dimension reduction, we can regard $\{v_i U_i^* f\}_{i \in I}$
as fusion frame measurements (cf. \cite{KPCL08}), where $U_i$ is a left-orthogonal basis for $\cW_i$,
i.e., $P_i = U_iU_i^*$ and $U_i^*U_i=I$. In this case, the reconstruction formula takes the form
\begin{equation} \label{eq:FFreconstrU}
f = \sum_{i \in I} v_i S^{-1} U_i (v_i U_i^* f).
\end{equation}
\end{remark}

\begin{remark}
\ap{Reconstruction of a sparse signal from its fusion frame measurements is considered in \cite{BKR08}.}
\end{remark}

\subsection{\ap{Applications of Fusion Frames}}

\ap{Frame theory has been established as a powerful mathematical framework for robust and stable representation of signals. It has found numerous applications in sampling theory \cite{E03}, data quantization \cite{BP07}, quantum measurements \cite{EF02}, coding \cite{BDV00,SH03}, image processing \cite{CD02,CLHB05}, wireless communications \cite{HBP01,HP02,S01}, time-frequency analysis \cite{DH02,DHRS03,WES05}, speech recognition \cite{BCE06}, and bioimaging \cite{CK08}. The reader is referred to survey papers \cite{KC07a,KC07b} and the references therein for more examples. Fusion frame theory is a generalization of frame theory that is more suited for applications where two-stage (local and global) signal/data analysis is required. To highlight this we give three signal processing applications wherein fusion frames arise naturally. We also discuss the connection between fusion frames and two pressing questions in pure mathematics.}

\ap{\textit{Distributed Sensing.} Consider} a large number of small and inexpensive sensors \ap{that} are deployed in \ap{an area} of interest to measure various physical quantities or to keep \ap{the area} under surveillance. Due to practical and economical factors, such as low communication bandwidth, limited signal processing power, limited battery life, or the topography of the surveillance area, the sensors are typically deployed in clusters, where each cluster includes a unit with higher computational and transmission power for local data processing. A typical large sensor network can thus be viewed as a redundant collection of subnetworks forming a set of subspaces (e.g. see \cite{CKLR07,KPCL08,PKC08}). The local subspace information are passed to a central processing station for joint processing. A similar local-global signal processing principle is applicable to modeling of human visual cortex as discussed in \cite{RJ06}.

\ap{\textit{Parallel Processing.}} If a frame system is simply too
large to handle effectively (from either computational complexity or
numerical stability standpoints), we can divide it into multiple
small subsystems for simple and perhaps parallelizable processing.
Fusion frames provide a natural framework for splitting a large
frame system into smaller subsystems and then recombining the
subsystems. Splitting of a large frame system into smaller
subsystems for parallel processing was first considered in
\cite{BM91} and predates the introduction of fusion frames.

\ap{\textit{Packet Encoding.}} Information bearing symbols are
typically encoded into a number of packets and then transmitted over
a communication network, e.g., the internet. The transmitted packet
may be corrupted during the transmission or completely lost due to
buffer overflows. By introducing redundancy in encoding the symbols,
we can increase the reliability of the communication scheme. Fusion
frames, as redundant collections of subspaces, can be used to
produce a redundant representation of a source symbol. In the
simplest form, each fusion frame projection can be viewed as a
packet that carries some new information about the symbol. The
packets can be decoded jointly at the destination to recover the
transmitted symbol. The use of fusion frames for packet encoding is
considered in \cite{Bod07}.

\ap{\textit{The Kadison-Singer Problem and Optimal Packings.}}
The Kadison-Singer Problem \cite{CT06} has been among the most
famous unsolved problems in analysis
 since 1959. \ap{It turns out that this problem is, roughly speaking,
 equivalent to the following question (cf. \cite{CT06}). Can a frame be
 partitioned such that the spans of
the partitions as a fusion frame lead to a `good' lower fusion frame
bound? The reader is referred to \cite{CT06} for details.} Therefore,
advances in the design of fusion frames will have direct impact in
providing new angles for a renewed attack to the Kadison-Singer
Problem. \ap{In addition, there is a close connection between
Parseval fusion frames and Grassmannian packings. In fact, as shown
in \cite{KPCL08},} Parseval fusion frames consisting of
equi-distance and equi-dimensional subspaces are optimal
Grassmannian packings. \ap{Therefore, new methods for constructing
such fusion frames also provide ways to construct optimal packings.
We note that the frame counterpart of this connection also exists (cf. \cite{SH03}).}

\subsection{\ap{Main Contribution: Construction of Fusion Frames with Desired Properties}}

\ap{The value of fusion frames for signal processing is that the
interplay between local-global processing and redundant
representation provides resilience to noise and erasures due to, for instance, sensor
failures or buffer overflows \cite{Bod07,CK07b,KPCL08,PKC08}. It also provides robustness to
subspace perturbations \cite{CKL08}, which may be due to imprecise
knowledge of sensor network topology. In most cases, extra
structure on fusion frames is requited to guarantee satisfactory
performance. For instance, our recent work \cite{KPCL08,PKC08} shows
that in order to minimize the mean-squared error in the linear
minimum mean-squared error estimation of a random vector from its
fusion frame measurements in white noise the
fusion frame needs to be Parseval or tight. The Parseval property is
also desirable for managing signal processing complexity.  It means
that the fusion frame operator $S$ is equal to the identity operator
and hence the operator inversion required for signal \ap{reconstruction is} trivial.
To provide maximal robustness against erasures of one fusion frame
subspace the fusion frame subspaces must also be equi-dimensional.
If maximal robustness with respect to two or more subspace erasures
is desired then the fusion frame subspaces must all have the same
pairwise chordal distance as well. Other examples of optimality of
\textit{structured} fusion frames for signal reconstruction can be
found in \cite{Bod07,CK07b,KPCL08,PKC08,CKL08}.}

\begin{remark}
We note that signal reconstruction in a frame system in the presence of erasures has been studied by several authors. The results indicate that robustness to erasures of frame coefficients also require the frame system to have specific properties and structure, such as Parseval, equiangular, or equal-norm property. The reader is referred to \cite{BOG08,BP05,CK03,GKK01,HP04,STDH07,TDHS05} and the references therein for a collection of relevant results.
\end{remark}

\ap{A natural question is: How can one construct fusion frames with
desired properties? More specifically, how can one construct fusion
frames for which a set of parameters such as}
\bitem
\item[1)] {\em eigenvalues of the fusion frame operator,}
\item[2)] {\em dimensions of the subspaces,}
\item[3)] {\em chordal distances between subspaces, \ap{and/or}}
\item[4)] {\em weights assigned to the subspaces}
\eitem
\ap{can be prescribed?}

\ap{In this paper, we present a complete answer to the above
question under the first design criterion and provide partial
answers for the construction of fusion frames under the second and
third design criteria. Our main contributions are as follow.

\bitem
\item In Section \ref{sec:characterization}, we provide a complete
characterization of \textit{Parseval} fusion frames in terms of the
existence of special isometries defined on an encompassing Hilbert
space.\vspace{0.1cm}
\item In Section \ref{sec:generalconstruction}, we present two general
ways for constructing a new fusion frame from a given fusion frame,
by exploiting the notions of spatial complement and Naimark
complement, and establish the relationship between the parameters of
the two fusion frames. In particular, we show how the weights,
subspace dimensions, fusion frame bounds, eigenvalues of the fusion
frame operator, and the chordal distance between the subspaces for
the new fusion frame can be determined from those of the original
fusion frame prior to construction.\vspace{0.1cm}
\item In Section \ref{sec:fusionframeoperator}, we establish existence
conditions and develop simple algorithms for constructing fusion
frames with \ap{desired} fusion frame operators.\footnote{\ap{Throughout this paper whenever we say a fusion frame with a desired fusion frame operator  we mean a fusion frame for which the fusion frame operator has a desired set of eigenvalues. A similar language is used to refer to a frame for which the frame operator has a desired set of eigenvalues.}} Our construction
produces frames with \ap{desired} frame operators as a special case.
\eitem }

\ap{We note that the construction of frames with arbitrary frame operators has been studied by several authors} (see, e.g.,  \cite{BF03,Cas04,CL06,CK07a}). However, the fusion frame counterparts are much less exploited. In fact, even establishing existence conditions for fusion frames is a deep and involved problem.  Frame potentials, introduced in (cf. \cite{BF03}), have proven to be a valuable tool
in asserting the existence of tight frames.  Two recent papers \cite{CF09,MRS08} have
introduced and studied {\em fusion frame potentials} to address the existence of fusion
frames, but with limited success.  The problem here is that minimizers of the fusion frame
potential are not necessarily tight fusion frames.  Also, the fusion frame potential is a very
complex notion and it requires some deep ideas to make it work.
  However, until recently, no general construction method was known for the construction of fusion frames \ap{with desired properties.  A significant advance for the construction of equi-dimensional
tight fusion frames
  was presented in \cite{CFMWZ09}. The authors have provided a complete characterization of triples $(M,N,m)$ for which tight fusion frames exist. Here $M$ is the total dimension of the Hilbert space, $N$ is the number of subspaces, and $m$ is the dimension of the fusion frame subspaces. They have also developed an elegant and simple algorithm which can produce a tight fusion frame for most $(M,N,m)$ triples.}

\ap{Our paper is concerned with a more general question than that answered in \cite{CFMWZ09}, that is, the construction of a fusion
frame (not necessarily tight) for which the fusion frame operator can possess any desired set of eigenvalues. This includes fusion
frames with desired bounds as a special case, as the fusion frame bounds are simply the smallest and largest eigenvalues of the
associated fusion frame operator. More specifically, given $M, N, m \in \NN$, and a set of real positive values $\{\lambda_j\}_{j=1}^M$,
we establish existence conditions for fusion frames whose fusion frame operators have eigenvalues $\{\lambda_j\}_{j=1}^M$ and
develop a simple algorithm that produces such a fusion frame. The answer to this problem has profound practical and theoretical implications. From a signal analysis standpoint, it \ap{provides a flexible} mathematical framework where the representation system can be tailored to satisfy data processing demands. From a theoretical standpoint,} it provides a deep understanding of the boundaries of fusion frame theory viewed as a generalization of frame theory.

\ap{We note that our solution provides an answer to the construction of frames with arbitrary frame operators as a special case, which has been an open problem until now. Construction of frames with arbitrary frame operators  is studied in \cite{CL02}
(See also \cite{Cas04})}

\section{Characterization of Parseval Fusion Frames}\label{sec:characterization}

In this section, we provide a characterization of \textit{Parseval}
fusion frames in terms of the existence of special isometries
defined on an encompassing Hilbert space. This characterization may
be viewed as the fusion frame counterpart to {\em Naimark's theorem}
\cite{Cas00,CHL00,Chr03,HL00}, where Parseval frames are characterized as frame
systems generated by an orthogonal projection of an orthonormal
basis from a larger Hilbert space. \ap{However, these
characterizations cannot be easily exploited for constructing
Parseval frames or Parseval fusion frames. The} difficulty arises
from the uncontrollable nature of the projection of the larger
Hilbert space. For fusion frames, the construction of appropriate
isometries are \ap{particularly} difficult.  In fact, these problems are
equivalent to serious unsolved problems in operator theory concerning
the construction of projections which sum to a given operator.
Nonetheless, these isometries are illuminating for understanding
Parseval fusion frames.

The following theorem states the main result of this section, which
can be regarded as a quantitative version of \cite[Thm. 3.1]{CKL08}.

\begin{theorem}
\label{theo:charac_PFF_intermsof_isom} For a complete family of
subspaces\footnote{A family of subspaces is called {\em complete} in
$\cH$, if their span equals $\cH$.} $\{\cW_i\}_{i \in I}$ of $\cH$
and positive weights $\{v_i\}_{i \in I}$, the following conditions
are equivalent. \bitem
\item[(i)] $\{(\cW_i,v_i)\}_{i\in I}$ is a Parseval fusion frame for $\cH$.
\item[(ii)] There exists a Hilbert space $\cK \supset \cH$, an orthonormal basis
$\{e_j\}_{j\in J}$ for $\cK$, a partition $\{J_i\}_{i\in I}$ of $J$, and isometries
$L_i : \cE_i := \spann\{e_j\}_{j\in J_i} \to \cW_i$, $i \in I$, such that
\[
P = \sum_{i\in I} v_i L_i
\]
is an orthogonal projection of $\cK$ onto $\cH$.
\eitem
\end{theorem}

\begin{proof}
(i) $\Rightarrow$ (ii). For every $i\in I$, let $\{e_{ij}\}_{j\in
J_i}$ be an orthonormal basis for $\cW_i$. Since
$\{(\cW_i,v_i)\}_{i\in I}$ is a Parseval fusion frame for $\cH$, by
\cite[Thm. 2.3]{CKL08}, the family $\{v_i e_{ij}\}_{i\in I, j\in
J_i}$ is a Parseval frame for $\cH$. This implies (cf. \cite{Cas00,Chr03,HL00})
that there exists a Hilbert space $\cK \supset \cH$ with an
orthonormal basis $\{\tilde{e}_{ij}\}_{i\in I,j\in J_i}$ so that the
orthogonal projection $P$ of $\cK$ onto $\cH$ satisfies
\[
P(\tilde{e}_{ij}) = v_i e_{ij}, \qquad i \in I,\: j \in J_i.
\]
Setting $\cE_i = \spann  \{\tilde{e}_{ij}\}_{j\in J_i}$, the map
\[
L_i := \frac{1}{v_i}P|_{\cE_i}:\cE_i \to \cW_i
\]
is an isometry for all $i\in I$, and
\[
P = \sum_{i\in I} v_i L_i
\]
is an orthogonal projection of $\cK$ onto $\cH$.

(ii) $\Rightarrow$ (i). Since $P=\sum_{i\in I} v_i L_i$ is an
orthogonal projection of $\cK$ onto $\cH$, $\{Pe_j\}_{j\in J}$ is a
Parseval frame for $\cH$. Further, since $L_i := 1/v_i \cdot
P|_{\cE_i}:\cE_i \to \cW_i$ is an isometry, it follows that $\{1/v_i
\cdot Pe_j\}_{j\in J_i}$ is an orthonormal basis for $\cW_i$, $i \in
I$. Applying these observations and denoting by $P_i$ the orthogonal
projection onto $\cW_i$, for all $f\in \cH$, we have
\begin{eqnarray*}
\sum_{i\in I} v_i^2 \|P_if\|^2
&=& \sum_{i\in I} v_i^2 \Big\|\sum_{j\in J_i}\Big\langle f,\frac{1}{v_i}Pe_j\Big\rangle
\frac{1}{v_i}Pe_j\Big\|^2\\
&=& \sum_{i\in I}v_i^2 \sum_{j\in J_i}|\Big\langle f, \frac{1}{v_i}Pe_j\Big\rangle |^2\\
&=& \sum_{i\in I}\sum_{j\in J_i}|\langle f,Pe_j \rangle |^2\\
&=& \|f\|^2.
\end{eqnarray*}
Thus $\{(\cW_i,v_i)\}_{i\in I}$ is a Parseval fusion frame as claimed.
\end{proof}

Considering this theorem and its proof, we can derive an
interesting corollary which links the construction of Parseval
fusion frames to the construction of special Parseval frames. In
fact, the question of existence of Parseval fusion frames is
equivalent to the question of existence of Parseval frames for which
certain subsets of \ap{frame vectors} are orthonormal. The answer to
\ap{this question} is not known, but the connection between the two
problems may provide insights into the construction Parseval fusion
frames.

\begin{corollary}
\label{coro:charac_PFF_intermsof_PF} For a family of subspaces
$\{\cW_i\}_{i \in I}$ of $\cH$ and positive weights $\{v_i\}_{i \in
I}$, the following conditions are equivalent.
\bitem
\item[(i)] $\{(W_i,v_i)\}_{i\in I}$ is a Parseval fusion frame for $\cH$.
\item[(ii)] There exists a Parseval frame $\{e_{ij}\}_{i\in I,j\in J_i}$ for
$\cH$ such that $\{1/v_i \cdot e_{ij}\}_{j\in J_i}$ is an
orthonormal basis for $\cW_i$ for all $i\in I$. \eitem
\end{corollary}


\section{\ap{Construction of New Fusion Frames from Existing Ones}}
\label{sec:generalconstruction}

In this section, \ap{we present two general ways, namely the spatial
complement and the Naimark complement, for constructing a new fusion
frame from a given fusion frame and establish the relationship
between the parameters of the two fusion frames. A special case of
the construction methods presented here is reported in
\cite{CFMWZ09}. The result of \cite{CFMWZ09} deals only with the
construction of Parseval fusion frames in a finite dimensional
Hilbert space and does not investigate the relation between the new
and the original fusion frame parameters.}

\subsection{\ap{The Spatial Complement}}

Taking the spatial complement \ap{appears to be a} natural way for
generating a new fusion frame from a given fusion frame. \ap{We
begin by defining the notion of an \textit{orthogonal fusion frame
to a given fusion frame}, which is central to our discussion.}

\begin{definition}
Let $\{(\cW_i,v_i)\}_{i \in I}$ be a fusion frame for $\cH$. If the
family $\{(\cW_i^{\perp},v_i)\}_{i\in I}$\ap{, where $W_i^{\perp}$
is the orthogonal complement of $W_i$,} is also a fusion frame,
\ap{then} we call \ap{$\{(\cW_i^{\perp},v_i)\}_{i\in I}$} the
orthogonal fusion frame to $\{(\cW_i,v_i)\}_{i \in I}$.
\end{definition}

\begin{theorem}\label{theo:orthoFF}
\ap{Let} $\{(\cW_i,v_i)\}_{i \in I}$ be a fusion frame for $\cH$
with optimal fusion frame bounds $0<A\le B < \infty$.  Then the
following conditions are equivalent. \vspace*{0.2cm}
\bitem
\item[(i)] $\bigcap_{i \in I} \cW_i= \{0\}$.\vspace*{0.2cm}
\item[(ii)] $B < \sum_{i \in I} v_i^2 $.\vspace*{0.2cm}
\item[(iii)] The family $\{(\cW_i^{\perp},v_i)\}_{i \in I}$ is a fusion frame
for $\cH$ with optimal fusion frame bounds $\sum_{i \in I} v_i^2 -B$
and $\sum_{i \in I} v_i^2 -A$.\vspace*{0.2cm} \eitem
\end{theorem}

\begin{proof}
(iii) $\Rightarrow$ (i):  Suppose that (i) is false. Then there
exists a vector $0\not= f\in \cap_{i \in I} \cW_i$. This implies $f
\perp \cW_i^{\perp}$ for all $i \in I$, hence $\{\cW_i^{\perp}\}_{i
\in I}$ does not span $\cH$. This is a contradiction to (iii).

(i) $\Rightarrow$ (ii): Since $B$ is optimal, by using the fusion
frame property, it follows that there exists some $f \in \cH$ so
that
\[
B\|f\|^2 = \Big\langle \sum_{i \in I} v_i^2 P_i f,f\Big\rangle
= \sum_{i \in I} v_i^2\|P_i f\|^2
\le \sum_{i  \in I} v_i^2 \|f\|^2.
\]
Hence
\begin{equation}\label{eq:Binequality}
B \le \sum_{i  \in I} v_i^2.
\end{equation}
It now suffices to observe that we have equality in \eqref{eq:Binequality} if and only if
\[
f \in \bigcap_{i  \in I} \cW_i \not= \{0\}.
\]

(ii) $\Rightarrow$ (iii):  Since $AI \le \sum_{i \in I} v_i^2P_i \le BI$, we have
\begin{equation}\label{eq:orthoPFF1}
\left ( \sum_{i \in I} v_i^2 -B\right ) I \le
\sum_{i \in I} v_i^2(I-P_i) \le \left ( \sum_{i \in I} v_i^2-A \right ) I.
\end{equation}
\ap{From} (ii), \ap{we have} $\sum_{i \in I} v_i^2 -B >0$ \ap{and} hence
\[
\{(\cW_i^{\perp},v_i)\}_{i \in I} = \{((I-P_i)\cH, v_i)\}_{i \in I},
\]
is a fusion frame. The fusion frame bounds from \eqref{eq:orthoPFF1} are optimal.
\end{proof}

\ap{The following theorem shows that all the parameters of the new fusion frame can be determined from those of the generating fusion frame prior to the construction.}

\begin{theorem}\label{theo:propertiesorthoFF}
Let $\{(\cW_i,v_i)\}_{i \in I}$ be a fusion frame for $\cH$, and let
$\{(\cW_i^{\perp},v_i)\}_{i \in I}$ be its associated orthogonal fusion frame.
Then the following conditions hold.
\bitem
\item[(i)] Let $S$ denote the frame operator for $\{(\cW_i,v_i)\}_{i \in I}$
with eigenvectors $\{e_j\}_{j \in J}$ and respective eigenvalues $\{\lambda_j\}_{j \in J}$.
Then the fusion frame operator for $\{(\cW_i^{\perp},v_i)\}_{i=1}^N$ possesses
the same eigenvectors $\{e_j\}_{j \in J}$ and respective eigenvalues
$\{\sum_{i \in I} v_i^2 - \lambda_j\}_{j \in J}$.
\item[(ii)] Assume that $\dim \cH < \infty$ and $m := \dim \cW_i$ for all $i \in I$. Then,
\[
d_c^2(\cW_i^{\perp},\cW_j^{\perp}) = d_c^2(\cW_i,\cW_j) + 2m-\dim \cH \quad \mbox{for all } i, j \in \{1,\ldots,N\}, \,i \neq j.
\]
\ap{where $d_c^2(\cW_i,\cW_j)$ denotes the \textit{squared chordal distance} between subspaces $\cW_i$ and $\cW_j$ and is given by
\[
d_c^2(\cW_i,\cW_j) = \dim \cH - \tr[P_i P_j].
\]}
\eitem
\end{theorem}

\begin{proof}
(i). For each $j \in J$, we have
\[
\sum_{i\in I} v_i^2 P_ie_j = \lambda_j e_j.
\]
Hence,
\[
\sum_{i\in I} v_i^2 (I-P_i)e_j = \left ( \sum_{i\in I} v_i^2 - \lambda_j \right ) e_j,
\]
which implies the claimed properties \ap{for} the fusion frame operator $S^{\perp}$.\\
(ii). The orthogonal projection \ap{onto $\cW_i^\perp$  is given} by $I-P_i$.
Hence,
\[
d_c^2(\cW_i^{\perp},\cW_j^{\perp}) = \dim \cH - \tr[(I-P_i)(I-P_j)].
\]
The claim follows from
\[
\tr[(I-P_i)(I-P_j)] = \tr[I-P_i-P_j+P_iP_j] = \dim \cH - 2 m + \tr[P_iP_j]
\]
and the definition of $d_c^2(\cW_i,\cW_j)$.
\end{proof}


\begin{corollary}\label{cor:Atight}
\ap{Let} $\{\cW_i\}_{i=1}^N$ be an $A$-tight fusion frame for
$\bR^M$ such that $\cW_{\ap{k}} \neq \cH$ for some $\ap{k} \in \{1,\ldots,N\}$. Then $\{W_i^{\perp}\}_{i=1}^N$
is an $(N-A)$-tight fusion frame for $\bR^M$. If $m := \dim \cW_i$ for all $i \in \{1,\ldots,N\}$
and $d^2 := d_c^2(\cW_i,\cW_j)$ for all $i, j \in \{1,\ldots,N\}$, $i \neq j$, then
\[
d_c^2(\cW_i^{\perp},\cW_j^{\perp}) = \ap{d^2} + 2m-M \quad \mbox{for all } i, j \in \{1,\ldots,N\}, \,i \neq j.
\]
\end{corollary}

\begin{proof}
Assume that $W_{\ap{k}}\not= \bR^M$. Then by choosing some $0\not= f\in W_{\ap{k}}^{\perp}$, we obtain
\[
A \|f\|^2 = \sum_{i=1}^{N}v_i^2\|P_if\|^2 = \sum_{i\not= \ap{k}}v_i^2\|P_if\|^2
< \Big(\sum_{i=1}^Nv_i^2\Big) \|f\|^2.
\]
Thus we have $A < \sum_{i=1}^Nv_i^2$, and the application of Theorem \ref{theo:orthoFF}
proves the first part of the claim. The second part follows immediately from Theorem \ref{theo:propertiesorthoFF} (ii).
\end{proof}

\ap{A straightforward application of Corollary \ref{cor:Atight}
provides a way of constructing tight fusion frames with
equi-dimensional subspaces. This construction starts with a given
set of equi-dimensional subspaces that do not form a tight fusion
frames and fills up the Hilbert space by adding a new set of
subspaces, with the same dimension, to produce a tight fusion
frame.}

\begin{corollary} \label{coro:waterfilling}
Let $\{\cW_i\}_{i=1}^{N}$ be a family of $m$-dimensional subspaces
of $\bR^M$. Then there exist $N(M-1)$ $m$-dimensional subspaces
$\{\cV_i\}_{i=1}^{N(M-1)}$ of $\bR^M$ so that
$\{\cW_i\}_{i=1}^{N}\cup\{\cV_i\}_{i=1}^{N(M-1)}$ is a tight fusion
\ap{frame. Moreover, if $N=1$ and $\dim \cW_1=M-1$ then the
construction is minimal in the sense that it identifies the smallest
number of $m$-dimensional subspaces which need to be added to obtain
a tight fusion frame.}
\end{corollary}

\begin{proof}
For each $i=1,\ldots,N$, we choose an orthonormal basis
$\{e_j^i\}_{j=1}^M$ for $\bR^M$ in such a way that
$\{e_j^i\}_{j=1}^m$ is an orthonormal basis for $\cW_i$. Let $T_i$,
$i=1,\ldots,N$, denote the \ap{circular} shift operator on the
orthonormal basis $\{e_j^i\}_{j=1}^M$. Then
\[ \{T_i^k \cW_i\}_{i=1,k=0}^{N\; ,M-1},\] is a tight fusion frame for $\bR^M$ of
$m$-dimensional subspaces which contains $\{\cW_i\}_{i=1}^{N}$.

\ap{Now consider the case where $N=1$ and $\dim \cW_1=M-1$.} Let
$\{\cV_i\}_{i=1}^{N_1}$ be any collection of ($M-1$)-dimensional
subspaces so that $\{\cW_1\}\cup \{\cV_i\}_{i=1}^{N_1}$ is a tight
fusion frame. By Theorem \ref{theo:orthoFF}, we have $1+N_1 =M$,
hence $N_1 = M-1$, which equals $N(M-1)$.
\end{proof}


\subsection{The Naimark Complement}

\ap{Another approach to constructing a new fusion frame from an
existing one is to use the notion of \textit{Naimark complement}.
This approach however applies to Parseval fusion frames only, as
stated in the following theorem.}

\begin{theorem}\label{theo:Naimark_PFF}
Let $\{(\cW_i,v_i)\}_{i \in I}$ be a Parseval fusion frame for
$\cH$. Then there exists a Hilbert space $\cK \supseteq \cH$ and a
Parseval fusion frame $\{(\cW_i',\sqrt{1-v_i^2})\}_{i \in I}$ for
$\cK \ominus \cH$ with the following properties.
\bitem
\item[(i)] $\dim \cW_i' = \dim \cW_i$ for all $i \in I$.
\item[(ii)] If $\dim \cH < \infty$ and $\dim \cW_i = \dim \cW_j$ for all $i, j \in I$, $i \neq j$, then
\[
d_c^2(\cW_i',\cW_j') = d_c^2(\cW_i,\cW_j) \quad \mbox{for all } i, j \in \{1,\ldots,N\}, \,i \neq j.
\]
\eitem
\end{theorem}

\begin{proof}
For each $i \in I$, let $\{f_{ij}\}_{j \in J_i}$ be an orthonormal basis for $\cW_i$. Then the family
\[
\{v_i f_{ij}\}_{i \in I, j \in J_i}
\]
is a Parseval frame for $\cH$.  By \cite{Cas00,Chr03,HL00}, there exists a
Hilbert space $\cK \supseteq \cH$, an orthogonal projection $P: \cK
\rightarrow \cH$, and an orthonormal basis $\{e_{ij}\}_{i \in I, j
\in J_i}$ for $\cK$ so that \beq \label{eq:relation_e_f} Pe_{ij} =
v_i f_{ij}, \qquad i \in I,\, j \in J_i. \eeq This implies that
$\{(I-P)e_{ij}\}_{i \in I, j \in J_i}$ is a Parseval frame for $\cK
\ominus \cH$. Further,
\[
\|(I-P)e_{ij}\| = \sqrt{1-v_i^2}, \qquad i \in I,\, j \in J_i,
\]
and, for $j, j' \in J_i$, $j \neq j'$, we have
\[
\langle (I-P)e_{ij},(I-P)e_{ij'}\rangle
= - \langle Pe_{ij},e_{ij'}\rangle
= - \langle v_i f_{ij},v_i f_{ij'}\rangle
= 0.
\]
Defining
\[
\cW_i' = \spann\{(I-P)e_{ij} : j \in J_i\},
\]
we conclude that $\{(\cW_i',\sqrt{1-v_i^2})\}_{i \in I}$ is a Parseval fusion frame for $\cK \ominus \cH$.\\
(i). By construction,
\[
\dim \cW_i' = |J_i| = \dim \cW_i \qquad \mbox{for all } i \in I.
\]
(ii). Set $M := \dim \cH$, $L := \dim \cK$, $I:= \{1,\ldots,N\}$,
and $m := \dim \cW_i$ for all $i \in \{1,\ldots,N\}$. For the sake
of brevity, we define $E_i := ((I-P)e_{i1}, \ldots, (I-P)e_{im}) \in
\RR^{M \times m}$ and $F_i := (v_i f_{i1}, \ldots, v_i f_{im}) \in
\RR^{M \times m}$. Then, for every $i, i' \in \{1,\ldots,N\}$, $i
\neq i'$, we obtain
\[
\tr[P_i P_{i'}]
= \tr[F_i F_i^T F_{i'} F_{i'}^T]
= \tr[(F_{i'}^T F_i) (F_i^T F_{i'})]
= \tr[(\ip{v_i f_{i'j}}{v_i f_{ik}})_{j,k} (\ip{v_i f_{ij}}{v_i f_{i'k}})_{j,k}].
\]
By employing \eqref{eq:relation_e_f},
\beq
\label{eq:tracefirstclaim} \tr[P_i P_{i'}] = \tr[(\ip{P e_{i'j}}{P
e_{ik}})_{j,k} (\ip{P e_{ij}}{P e_{i'k}})_{j,k}].
\eeq
Now letting $P_i'$ denote the orthogonal projection onto $\cW_i'$, for each $i,
i' \in \{1,\ldots,N\}$, $i \neq i'$, the definition of $\cW_i'$
implies
\[
\tr[P_i' P_{i'}']
= \tr[E_i E_i^T E_{i'} E_{i'}^T]
= \tr[(E_{i'}^T E_i) (E_i^T E_{i'})]
\]
and
\[
(E_{i'}^T E_i)
= \ap{( \ip{(I-P)e_{i'j}}{(I-P)e_{ik}})_{j,k}.}
\]
Utilizing the choice of $\{e_{ij}\}$ and careful dealing with the inner products on $\cK, \cH$, and $\cK \ominus \cH$, for each $j, k$,
\ap{\[
\ip{(I-P)e_{i'j}}{(I-P)e_{ik}}
= \ip{e_{i'j}}{e_{ik}} - \ip{Pe_{i'j}}{Pe_{ik}}
= - \ip{Pe_{i'j}}{Pe_{ik}}.
\]}
Combining the above three equations,
\[
\tr[P_i' P_{i'}'] = \tr[(\ip{Pe_{i'j}}{Pe_{ik}})_{j,k} (\ip{Pe_{ij}}{Pe_{i'k}})_{j,k}].
\]
Comparison with \eqref{eq:tracefirstclaim} completes the proof.
\end{proof}

\begin{definition}
Let $\{(\cW_i,v_i)\}_{i \in I}$ be a tight fusion frame for $\cH$.
We refer to the tight fusion frame $\{(\cW_i',\sqrt{1-v_i^2})\}_{i
\in I}$ for $\cK \ominus \cH$ from Theorem \ref{theo:Naimark_PFF} as
\ap{the Naimark fusion frame associated with $\{(\cW_i,v_i)\}_{i \in
I}$. The rationale for this terminology is} that this is the
fusion frame version of  the Naimark theorem \cite{Cas00,Chr03,HL00}.
\end{definition}

\begin{corollary}
\ap{Let} $\{\cW_i\}_{i=1}^N$ be an $A$-tight fusion frame for
$\bR^M$. Then there exists some $L \ge M$ and a
$\sqrt{1-1/A^2}$-tight fusion frame for $\bR^{L-M}$ which satisfies
$\dim \cW_i' = \dim \cW_i$ for all $i \in \{1,\ldots,N\}$. If, in
addition, $d^2 := d_c^2(\cW_i,\cW_j)$ for all $i, j \in
\{1,\ldots,N\}$, $i \neq j$, then
\[
d_c^2(\cW_i',\cW_j') = d^2 \quad \mbox{for all } i, j \in \{1,\ldots,N\}, \,i \neq j.
\]
\end{corollary}

\begin{proof}
This follows immediately from Theorem \ref{theo:Naimark_PFF}.
\end{proof}

We \ap{note} that Theorem \ref{theo:Naimark_PFF} is \ap{not always
constructive}, since it requires the knowledge of a larger Hilbert
space from which the given Parseval frame is derived by an
orthogonal projection of an orthonormal basis.


\section{Existence and Construction of \ap{A} Fusion \ap{Frame} with \ap{A Desired} Fusion Frame Operator}
\label{sec:fusionframeoperator}

\ap{We now focus on the existence and construction of fusion frames
whose fusion frame operators possess a \ap{desired} set of
eigenvalues. We answer the following questions: (1) Given a set of
eigenvalues, does there exist a fusion frame whose fusion frame
operator possesses those eigenvalues? (2) If such a fusion frame
exists how can it be constructed?}

\ap{Let} $\lambda_1 \ge \ldots \ge \lambda_M > 0$, $M \in \NN$, be real positive values satisfying a factorization as
\[
\hspace*{-2cm} \mbox{{\sc (Fac)}} \hspace*{2cm} \sum_{j=1}^M \lambda_j = Nm \in \NN.
\]
We wish to construct a fusion frame $\{\cW_i\}_{i=1}^N$, $\cW_i \subseteq \bR^M$, such that
\bitem
\item[{\sc (FF1)}] $\dim \cW_i = m$ for all $i=1\ldots,N$, and
\item[{\sc (FF2)}] the associated fusion frame operator has $\{\lambda_j\}_{j=1}^M$ as its eigenvalues.
\eitem

\subsection{The Integer Case}

\ap{We first consider the simple case where $\lambda_j \in \NN$ for
all $i=1,\ldots,M$. This case is central to developing intuition
about the construction algorithms to be developed.}

\begin{proposition}
\ap{If} the positive integers $N \ge
\lambda_1\geq\lambda_2\geq\cdots\geq\lambda_M>0$, $N \in \NN$, and
$m \in \NN$ satisfy {\sc (Fac)}, then the fusion frame
$\{\cW_i\}_{i=1}^N$ constructed via the {\sc (FFCIE)} algorithm
outlined in Figure \ref{fig:integerfusionframealgorithm} satisfies
both {\sc (FF1)} and {\sc (FF2)}.
\end{proposition}

\begin{figure}[h]
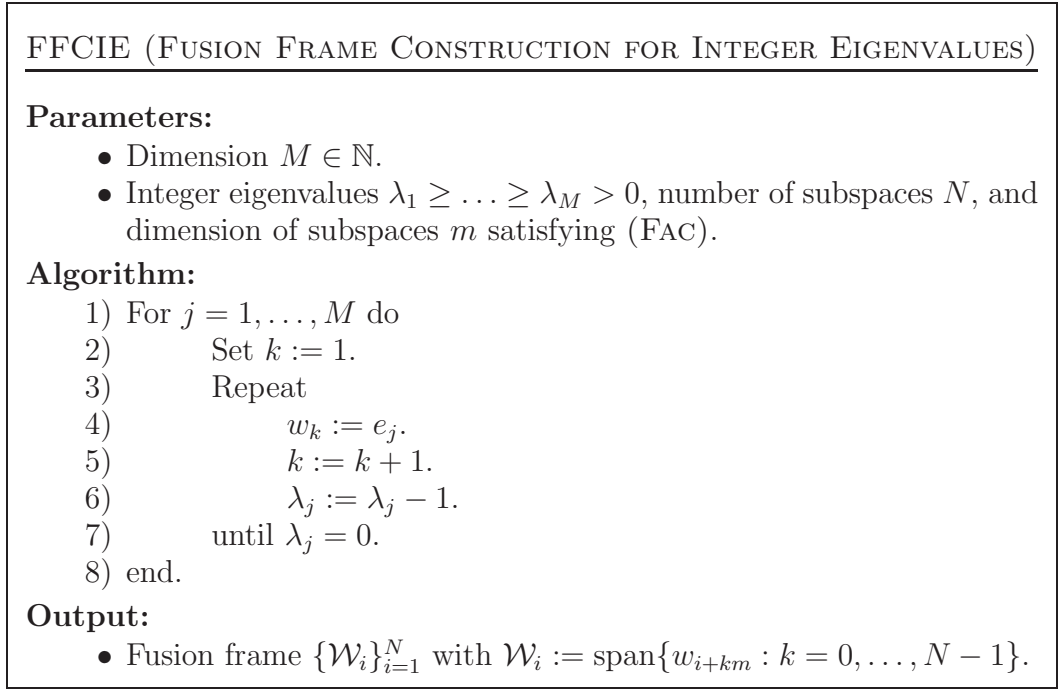

\centering
\framebox{
\begin{minipage}[h]{5.3in}
\vspace*{0.3cm}
{\sc \underline{FFCIE (Fusion Frame Construction for Integer Eigenvalues)}}

\vspace*{0.4cm}

{\bf Parameters:}\\[-3ex]
\begin{itemize}
\item Dimension $M \in \NN$.
\item Integer eigenvalues $\lambda_1 \ge \ldots \ge \lambda_M > 0$, number of subspaces $N$,
and dimension of subspaces $m$ satisfying {\sc (Fac)}.
\end{itemize}

{\bf Algorithm:}\\[-3ex]
\begin{itemize}
\item[1)] For $j=1,\ldots,M$ do
\item[2)] \hspace*{1cm} Set $k := 1$.
\item[3)] \hspace*{1cm} Repeat
\item[4)] \hspace*{2cm} $w_k := e_j$.
\item[5)] \hspace*{2cm} $k := k+1$.
\item[6)] \hspace*{2cm} $\lambda_j := \lambda_j - 1$.
\item[7)] \hspace*{1cm} until $\lambda_j = 0$.
\item[8)] end.
\end{itemize}

{\bf Output:}\\[-3ex]
\begin{itemize}
\item Fusion frame $\{\cW_i\}_{i=1}^N$ with $\cW_i := \spann\{w_{i+km} : k = 0,\ldots,N-1\}$.
\end{itemize}
\vspace*{0.01cm}
\end{minipage}
}
\caption{\ap{The {\sc FFCIE}} Algorithm \ap{for constructing} a fusion
frame with \ap{a} fusion frame operator \ap{with prescribed integer
eigenvalues}.} \label{fig:integerfusionframealgorithm}
\end{figure}

\begin{proof}
If the set of vectors
\[
\{w_{i+km} : k = 0,\ldots,N-1\}
\]
is pairwise orthogonal for each $i = 1,\ldots,N$, then {\sc (FF1)}
and {\sc (FF2)} follow automatically. Now fix $i \in
\{1,\ldots,N\}$. By construction, it is sufficient to show
that, for each $0 \le k \le N-2$, the vectors $w_{i+km}$ and
$w_{i+(k+1)m}$ are orthogonal. Again by construction, the only
possibility for this to fail is that there exists some $j_0 \in
\{1,\ldots,M\}$ satisfying $\lambda_{i_0} > N$. But this was
excluded by the hypothesis.
\end{proof}

\ap{The algorithm \ap{outlined} in Figure
\ref{fig:integerfusionframealgorithm} shuffles the intended
eigenvalues in terms of associated unit vectors $e_1, \ldots, e_M
\in \bR^M$ as basis vectors into the subspaces of the fusion frame
to be constructed.} Considering \ap{a} matrix $W \in \RR^{Nm \times
M}$ with the vectors $w_1,\ldots,w_{Nm}$ as rows, intuitively {\sc
(FFCIE)} fills this matrix up from top to bottom, row by row in such
\ap{a} way that the $\ell_2$ norm of the rows is $1$, the $\ell_2$
norm of column $j$ is $\lambda_j$, $j=1,\ldots,M$, and the
columns are orthogonal. The vectors $w_k$ are then assigned to
subspaces in such a way that the vectors assigned to each subspace
forms an orthonormal system. We note that the generated vectors
$w_k$, $k=1\ldots,Nm$ are as sparse as possible, providing optimal
fast computation abilities.

We wish to note that the condition $N \ge \lambda_1$ is necessary
for \ap{{\sc (FFCIE)}}. The question whether \ap{or not} this is necessary in
general is much more involved and will not be discussed here.

\subsection{The General Case}

\ap{We now discuss the general case where the desired eigenvalues
for the fusion frame operator are real positive values that satisfy
{\sc (Fac)}.}

\subsubsection{The Algorithm}

As a first step we generalize {\sc (FFCIE)} (see Figure
\ref{fig:integerfusionframealgorithm}) by introducing Lines 4) --
9), which deal with the non-integer parts. The \ap{construction}
algorithm for real eigenvalues\ap{, called {\sc (FFCRE)},} is
\ap{outlined} in Figure \ref{fig:fusionframealgorithm}.

\begin{figure}[h]
\centering
\framebox{
\begin{minipage}[h]{5.3in}
\vspace*{0.3cm}
{\sc \underline{FFCRE (Fusion Frame Construction for Real Eigenvalues)}}

\vspace*{0.4cm}

{\bf Parameters:}\\[-3ex]
\begin{itemize}
\item Dimension $M \in \NN$.
\item Eigenvalues $\lambda_1 \ge \ldots \ge \lambda_M > 0$, number of subspaces $N$, and dimension of subspaces $m$
satisfying {\sc (Fac)}.
\end{itemize}

{\bf Algorithm:}\\[-3ex]
\begin{itemize}
\item[1)] For $j=1,\ldots,M$ do
\item[2)] \hspace*{1cm} Set $k := 1$.
\item[3)] \hspace*{1cm} Repeat
\item[4)] \hspace*{2cm} If $\lambda_j < 2$ and $\lambda_j \neq 1$ then
\item[5)] \hspace*{3cm} $w_k := \sqrt{\frac{\lambda_j}{2}} \cdot e_j + \sqrt{1-\frac{\lambda_j}{2}} \cdot e_{j+1}$.
\item[6)] \hspace*{3cm} $w_{k+1} := \sqrt{\frac{\lambda_j}{2}}\cdot e_j - \sqrt{1-\frac{\lambda_j}{2}} \cdot e_{j+1}$.
\item[7)] \hspace*{3cm} $k := k+2$.
\item[8)] \hspace*{3cm} $\lambda_j := 0$.
\item[9)] \hspace*{3cm} $\lambda_{j+1} := \lambda_{j+1} - (2-\lambda_j)$.
\item[10)] \hspace*{2cm} else
\item[11)] \hspace*{3cm} $w_k := e_j$.
\item[12)] \hspace*{3cm} $k := k+1$.
\item[13)] \hspace*{3cm} $\lambda_j := \lambda_j - 1$.
\item[14)] \hspace*{2cm} end;
\item[15)] \hspace*{1cm} until $\lambda_j = 0$.
\item[16)] end;
\end{itemize}

{\bf Output:}\\[-3ex]
\begin{itemize}
\item Fusion frame $\{\cW_i\}_{i=1}^N$ with $\cW_i := \spann\{w_{i+km} : k = 0,\ldots,N\}$.
\end{itemize}
\vspace*{0.01cm}
\end{minipage}
}
\caption{\ap{The FFCRE} algorithm \ap{for constructing} a fusion frame with \ap{a desired fusion frame operator.}}
\label{fig:fusionframealgorithm}
\end{figure}

The principle for constructing the \ap{row} vectors \ap{$w_k$} which
generate the subspaces \ap{$\cW_i$} of the fusion frame is similar
to \ap{that in} {\sc (FFCIE)}, that is, again the
matrix $W$ which contains the vectors $w_k$, $k=1\ldots,Nm$ as rows
is filled up from top to bottom, row by row in such as way that the
$\ell_2$ norm of the rows is $1$, the $\ell_2$ norm of column
$j$ is $\lambda_j$, $j=1,\ldots,M$, and the columns are orthogonal.
The vectors $w_k$ are then assigned to subspaces in such a way that
the vectors assigned to each subspace form an orthonormal system.
However, here the task is more delicate since the $\lambda_j$'s are
not all integers. This forces the introduction of ($2 \times
2$)-submatrices of the type
\[
\left( \begin{array}{cc}
\sqrt{\frac{\lambda_j}{2}} & \sqrt{1-\frac{\lambda_j}{2}}\\
\sqrt{\frac{\lambda_j}{2}} & -\sqrt{1-\frac{\lambda_j}{2}}
\end{array} \right).
\]
\ap{These submatrices have orthogonal columns and unit norm ($\ell_2$ norm) rows and allow us to handle non-integer eigenvalues.}  This construction was introduced in \cite{CFMWZ09}
for constructing tight fusion frames.

\ap{Before we prove that {\sc (FFCRE)} indeed produces fusion frames
with desired operators we consider a special case, in which the
construction coincides with the construction of frames with
desired frame operators. Our intention is to highlight the
applicability of {\sc (FFCRE)}  to the construction of frames with
arbitrary frame operators and to present a simple example
that demonstrates how the algorithm works. A detail analysis of the
algorithm and the proof of its correctness are provided in
Subsection \ref{sec:MainAnalysis}}.


\subsubsection{A Special Case and An Example}

\ap{In the special case where $m=1$ a fusion frame reduces to a
frame and {\sc (FFCRE)} simplifies to an algorithm for constructing
frames with desired fusion frame operators. This algorithm, which
we refer to as {\sc (FCRE)}, is outlined in Figure
\ref{fig:framealgorithm}.}

\begin{figure}[h]
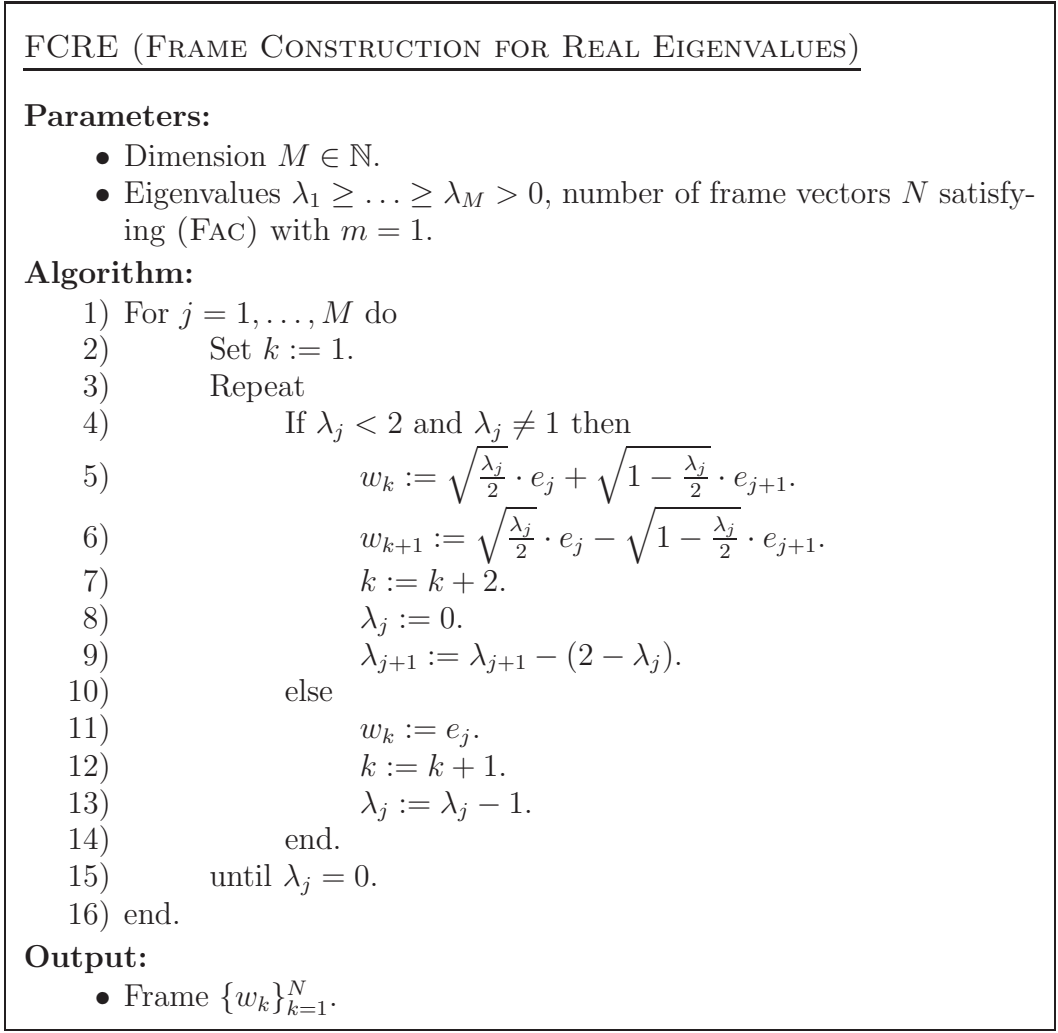

\centering
\framebox{
\begin{minipage}[h]{5.3in}
\vspace*{0.3cm}
{\sc \underline{FCRE (Frame Construction for Real Eigenvalues)}}

\vspace*{0.4cm}

{\bf Parameters:}\\[-3ex]
\begin{itemize}
\item Dimension $M \in \NN$.
\item Eigenvalues $\lambda_1 \ge \ldots \ge \lambda_M > 0$, number of frame vectors $N$
satisfying {\sc (Fac)} with $m=1$.
\end{itemize}

{\bf Algorithm:}\\[-3ex]
\begin{itemize}
\item[1)] For $j=1,\ldots,M$ do
\item[2)] \hspace*{1cm} Set $k := 1$.
\item[3)] \hspace*{1cm} Repeat
\item[4)] \hspace*{2cm} If $\lambda_j < 2$ and $\lambda_j \neq 1$ then
\item[5)] \hspace*{3cm} $w_k := \sqrt{\frac{\lambda_j}{2}} \cdot e_j + \sqrt{1-\frac{\lambda_j}{2}} \cdot e_{j+1}$.
\item[6)] \hspace*{3cm} $w_{k+1} := \sqrt{\frac{\lambda_j}{2}}\cdot e_j - \sqrt{1-\frac{\lambda_j}{2}} \cdot e_{j+1}$.
\item[7)] \hspace*{3cm} $k := k+2$.
\item[8)] \hspace*{3cm} $\lambda_j := 0$.
\item[9)] \hspace*{3cm} $\lambda_{j+1} := \lambda_{j+1} - (2-\lambda_j)$.
\item[10)] \hspace*{2cm} else
\item[11)] \hspace*{3cm} $w_k := e_j$.
\item[12)] \hspace*{3cm} $k := k+1$.
\item[13)] \hspace*{3cm} $\lambda_j := \lambda_j - 1$.
\item[14)] \hspace*{2cm} end.
\item[15)] \hspace*{1cm} until $\lambda_j = 0$.
\item[16)] end.
\end{itemize}

{\bf Output:}\\[-3ex]
\begin{itemize}
\item Frame $\{w_k\}_{k=1}^N$.
\end{itemize}
\vspace*{0.01cm}
\end{minipage}
}
\caption{The \ap{FCRE} algorithm \ap{for constructing a} frame with \ap{a desired} frame operator.}
\label{fig:framealgorithm}
\end{figure}

We now present \ap{an} example to demonstrate \ap{the application of
{\sc (FCRE)} as a special case of {\sc (FFCRE)}.}

\begin{example}
\ap{Let} $M=3$, $m=1$ (special case of frame construction), $N=8$,
and $\lambda_1=\frac{11}{4}$, $\lambda_2=\frac{11}{4}$, $\lambda_3=\frac{10}{4}$. Then, the algorithm
constructs the following \ap{matrix $W$}. Notice that indeed the $\ell_2$
norm of the rows is $1$, the $\ell_2$ norm of the column $j$ is $\lambda_j$, $j=1,\ldots,M$, and
the columns are orthogonal.
\[ W =
\begin{bmatrix}
 1 & 0&0 \\
 1&0&0\\
\sqrt{3/8}& \sqrt{5/8}&0\\
\sqrt{3/8} & -\sqrt{5/8} & 0\\
 0 & 1 & 0\\
 0 & \sqrt{1/4} & \sqrt{3/4}\\
 0 & \sqrt{1/4} & -\sqrt{3/4}\\
0&0&1
\end{bmatrix}
\]
The eigenvalues of the frame operator of the constructed frame
$\{w_{k,\cdot}\}_{k=1}^8$ are indeed $\frac{11}{4}$, $\frac{11}{4}$,
and $\frac{10}{4}$ as a simple computation shows. This also follows
from Theorem \ref{theo:givenFFOperator1} or Corollary
\ref{coro:givenFFOperator1} presented later in this subsection.
\end{example}

From now on we concentrate on the analysis of {\sc (FFCRE)}, keeping
in mind that our analysis also applies to {\sc (FCRE)} as a special
case.

\subsubsection{Feasibility Checks}

Before \ap{proving that} {\sc (FFCRE)} indeed produces a fusion
frame satisfying {\sc (FF1)} and {\sc (FF2)}, \ap{we investigate the
feasibility of the solution furnished by the algorithm.}

\begin{lemma}
\label{lemma:norm}
For all $k=1,\ldots,Nm$,
\[
\|w_k\|_2^2 = 1.
\]
\end{lemma}

\begin{proof}
This follows immediately from Lines 5), 6), and 11) of {\sc (FFCRE)}.
\end{proof}

\ap{Denoting} the $\lambda_j$'s in Lines 4) -- 6) of {\sc (FFCRE)} by $\tilde{\lambda}_j$'s to distinguish them from the eigenvalues
$\lambda_j$, $j=1\ldots,M$, the only two problems which could occur while running {\sc (FFCRE)} are:
\bitem
\item[{\sc (P1)}] $\lambda_{j+1} - (2-\tilde{\lambda}_j) < 0$ in Line 9) for some $j=1,\ldots,M-1$,
\item[{\sc (P2)}] using $e_{M+1}$ in Lines 5) -- 9) when performing the step for $j=M$.
\eitem
The following result shows that these cannot happen.

\begin{proposition}
\label{prop:P1P2}
If $\lambda_j \ge 2$ for all $j=1\ldots,M$, then {\sc (P1)} and {\sc (P2)} cannot happen.

\end{proposition}

\begin{proof}
{\sc (P1)}. Since $\lambda_j \ge 2$ for all $j=1\ldots,M$, we have
\[
\lambda_{j+1} \ge 2 \ge 2 - \tilde{\lambda}_{j} \quad \mbox{for all }j=1,\ldots,M-1.
\]
{\sc (P2)}. Suppose the algorithm is executed until Line 16) with $j=M-1$. Let $K+1$ denote the value which $k$
has reached at this point, and denote the coefficients of the vectors $w_k$ by $w_k =(w_{k1}, \ldots, w_{kM})$.
\ap{This} means that so far we have constructed $w_{kj}$ for $k=1, \ldots, K$, $j=1\ldots,M-1$. Then, by construction,
\beq \label{eq:feasible1}
\sum_{k=1}^{K} w_{kj}^2 = \lambda_j \quad \mbox{for all } 1 \le j \le M-1.
\eeq
\ap{We have} to distinguish between two cases:

{\em Case 1.} $w_{K-2,M} = 0$ and $w_{K-1,M} = 0$. Then, by \eqref{eq:feasible1} and Lemma \ref{lemma:norm},
\[
\sum_{j=1}^{M-1} \lambda_j
= \sum_{j=1}^{M-1}\sum_{k=1}^{K} w_{kj}^2
= \sum_{k=1}^{K} \sum_{j=1}^{M-1} w_{kj}^2
= \sum_{k=1}^{K} 1
= K.
\]
Since
\[
\sum_{j=1}^M \lambda_j = \lambda_M + \sum_{j=1}^{M-1} \lambda_j = \lambda_M + K\]
is an integer, it follows that $\lambda_M$ is an integer as well. Hence during the step
$j=M$ only the Block 11) -- 14)  as opposed to the Block 5) -- 9) will be executed.
Thus {\sc (P2)} does not happen.

{\em Case 2.} $w_{K-2,M} = \sqrt{1-\frac{\tilde{\lambda}_{M-1}}{2}}$ and $w_{K-1,M} = -\sqrt{1-\frac{\tilde{\lambda}_{M-1}}{2}}$.
In this case,
\[
\sum_{j=1}^{M-1} \lambda_j + (2- \tilde{\lambda}_{M-1})
= \sum_{j=1}^{M-1}\sum_{k=1}^{K} w_{kj}^2
= \sum_{k=1}^{K} \sum_{j=1}^{M-1} w_{kj}^2
= \sum_{k=1}^{K} 1
= K,
\]
an integer. Since $\sum_{j=1}^M \lambda_j$ is an integer as well, so is
\[
\sum_{j=1}^M \lambda_j - \left( \sum_{j=1}^{M-1} \lambda_j + (2-\tilde{\lambda}_{M-1})\right)
= \lambda_M - (2-\tilde{\lambda}_{M-1}).
\]
Hence, as before, in the step $j=M$ only the Block 11) -- 14)  as opposed to the Block 5) -- 9) will be
executed; here $ \lambda_M - (2-\tilde{\lambda}_{M-1})$ times.
Thus, in this situation, {\sc (P2)} does not occur.
\end{proof}

\subsubsection{Terminology and Lemmata}

In preparation for \ap{a detailed} analysis of \ap{{\sc
(FFCRE)}}, \ap{which is presented in Subsection
\ref{sec:MainAnalysis}}, we need to establish some terminology and a
few results.

\begin{definition}
An entry of a vector $w_k$, $k \in \{1,\ldots,Nm\}$ of the form $\pm \sqrt{1-\frac{\tilde{\lambda}_j}{2}}$
(entered in Line 5) or 6) of {\sc (FFCRE)}) will be termed a {\em terminal point}. An
{\em initial point} will be an entry of the form \ap{$\pm \sqrt{\tilde{\lambda}_j/2}$}
(entered in Line 5) or 6)).
\end{definition}

Considering the matrix $W \in \RR^{Nm \times M}$ with the vectors $w_1,\ldots,w_{Nm}$ as
rows, the initial points start non-zero entries in a row with more than one non-zero
entry, whereas the terminal points end such non-zero entries. It is obvious from algorithm
{\sc (FFCRE)} that column $n$ of $W$ has no initial points if and only if
\[
\sum_{j=1}^n \lambda_j \mbox{ is an integer},
\]
and it has no terminal points if and only if
\[
\sum_{j=1}^{n-1}\lambda_j \mbox{ is an integer}.
\]

\ap{Let $N(j)$ denote} the number of non-zero terms in each column
$j$, $j = 1,\ldots,M$ of the matrix $W$, that is, \ap{let} $N(j)$
denote\  the number of non-zero entries of the vector $w_{\cdot,j}$.
The following proposition determines exactly the value of $N(j)$
\ap{depending} on the occurrence of initial and/or terminal points.
\ap{We remind} the reader of the definition of $\tilde{\lambda}_j$
right before Proposition \ref{prop:P1P2}.

\begin{lemma}\label{lemma:Nj}
The following conditions hold for the previously defined values $N(j)$, $j = 1,\ldots,M$.
\bitem
\item[(i)] $N(j) = \lambda_j$, if $w_{\cdot,j}$ contains no initial or terminal points.
\item[(ii)] $N(j) = \lfloor \lambda_j \rfloor +1$, if $w_{\cdot,j}$ contains terminal, but no initial points.
\item[(iii)] $N(j) = \lfloor \lambda_j \rfloor +2$, if $w_{\cdot,j}$ contains initial, but no terminal points.
\item[(iv)] If $w_{\cdot,j}$ contains both initial and terminal points, then
\bitem
\item[(a)] if $\tilde{\lambda}_j \ge \tilde{\lambda}_{j-1}$, then $N(j) = \lfloor \lambda_j \rfloor +2$,
\item[(b)] if $\tilde{\lambda}_j < \tilde{\lambda}_{j-1}$ then $N(j) = \lfloor \lambda_j \rfloor +3$.
\eitem
\item[(v)] If $\lambda_{j_0}$ is the first non-integer value, then $N(j_0) = \lfloor \lambda_{j_0} \rfloor +2$.
\item[(vi)] If $\lambda_{j_1}$ is the last non-integer value, then $N(j_1) = \lfloor \lambda_{j_1} \rfloor +1$.
\eitem
\end{lemma}

\begin{proof}
(i). This is obvious, since in this case only the Block 11) -- 14) is performed as opposed to the
Block 5) -- 9).\\
(ii). Letting $n_j$ denote the number of ones in the vector $w_{\cdot,j}$, it follows that
$N(j) = n_j + 2$. Hence
\[
\lambda_j = n_j + (2-\tilde{\lambda}_{j-1}) = n_j + 1 + (1-\tilde{\lambda}_{j-1}) \quad \mbox{with }
0 < 1-\tilde{\lambda}_{j-1} < 1.
\]
This implies $\lfloor \lambda_j \rfloor = n_j+1$, and thus
\[
\lfloor \lambda_j \rfloor +1 = n_j+2 = N(j).
\]
(ii). Letting $n_j$ denote the number of ones in the vector $w_{\cdot,j}$, it follows that
$N(j) = n_j + 2$. Since the entries of $w_{\cdot,j}$ are $n_j$ times a $1$ as well as
the values $\pm \sqrt{1-\frac{\tilde{\lambda}_{j-1}}{2}}$,
\[
\lambda_j = n_j + (2-\tilde{\lambda}_{j-1}) = n_j + 1 + (1-\tilde{\lambda}_{j-1}) \quad \mbox{with }
0 < 1-\tilde{\lambda}_{j-1} < 1.
\]
This implies $\lfloor \lambda_j \rfloor = n_j+1$, and thus
\[
N(j) = n_j+2 = \lfloor \lambda_j \rfloor +1.
\]
(iii). Now the non-zero entries of the vector $w_{\cdot,j}$ are $\pm \sqrt{\frac{\tilde{\lambda}_j}{2}}$
as well as $n_j$, say, entries $1$. Hence $N(j) = n_j + 2$, and
\[
\lambda_j = n_j + \tilde{\lambda}_{j} \quad \mbox{with } 0 < \tilde{\lambda}_{j} < 1.
\]
This implies $\lfloor \lambda_j \rfloor = n_j$, and thus
\[
N(j) =  n_j+2 = \lfloor \lambda_j \rfloor +2.
\]
(iv). The vector $w_{\cdot,j}$ contains as non-zero entries, the initial points $\pm \sqrt{\frac{\tilde{\lambda}_j}{2}}$
and the terminal points $\pm \sqrt{1-\frac{\tilde{\lambda}_{j-1}}{2}}$ as well as, say, $n_j$ entries $1$, hence
$N(j) = kn_j +4$. Thus
\[
\lambda_j = n_j + \tilde{\lambda}_{j} + (2-\tilde{\lambda}_{j-1})
= n_j + 2 + (\tilde{\lambda}_{j-1}-\tilde{\lambda}_{j}).
\]
If $\tilde{\lambda}_{j-1}-\tilde{\lambda}_{j}\ge 0$, then $\lfloor \lambda_j \rfloor = n_j + 2$, which implies
\[
N(j) = n_j +4 = \lfloor \lambda_j \rfloor +2.
\]
If $\tilde{\lambda}_{j-1}-\tilde{\lambda}_{j}< 0$, then $\lfloor \lambda_j \rfloor = n_j + 1$, which implies
\[
N(j) = n_j +4 = \lfloor \lambda_j \rfloor +3.
\]
(v) \ap{and} (vi). These are direct consequences from the previous conditions.
\end{proof}

\ap{The following lemma shows an interesting relation between
consecutive values of $N(j)$ as $j$ progresses. However, we note
that} only the previous lemma is required for the proofs of the main
theorems \ap{that will be presented in Subsection
\ref{sec:MainAnalysis}.}

\begin{lemma}
For any $j \in \{1,\ldots,M-1\}$, the following conditions hold for the previously
defined values $N(j)$ and $N(j+1)$ supposing that they are not integers.
\bitem
\item[(i)] If $w_{\cdot,j}$ contains no initial or terminal points, then $N(j) \ge N(j+1)-1$.
\item[(ii)] If $w_{\cdot,j}$ contains initial, but no terminal points, then
\bitem
\item[(a)] if $\lambda_j + \lambda_{j+1}$ is an integer, then $N(j) \ge N(j+1)+1$,
\item[(b)] if $\lambda_j + \lambda_{j-1}$ is not an integer, then $N(j) \ge N(j+1)-1$.
\eitem
\item[(iii)] If $w_{\cdot,j}$ contains both initial and terminal points, then $N(j) \ge N(j+1)-1$.
\eitem
\end{lemma}

\begin{proof} Recall that we have $\lambda_j\geq\lambda_{j+1}$.\\
(i). Since $w_{\cdot,j}$ contains no initial points and $\lambda_{j+1}$ is not an integer,
the vector $w_{\cdot,j+1}$ contains initial, but no terminal points. Thus, by Lemma \ref{lemma:Nj},
\[
N(j)=\lfloor \lambda_j \rfloor +1 \geq\lfloor \lambda_{j+1} \rfloor +2-1=N(j+1)-1.
\]
(ii). By Lemma \ref{lemma:Nj}, $N(j)=\lfloor \lambda_j \rfloor +2$. Also $w_{\cdot,j}$
contains initial points, hence the vector $w_{\cdot,j+1}$ contains terminal points.\\
(a). Since $\lambda_j+\lambda_{j+1}$ is an integer and $w_{\cdot,j}$ does not contain
any terminal points, the vector $w_{\cdot,j+1}$ does not contain initial points.
This implies $N(j+1)=\lfloor \lambda_{j+1} \rfloor +1$.\\
(b). Since $\lambda_j+\lambda_{j+1}$ is not an integer and $w_{\cdot,j}$ does not contain
any terminal points, the vector $w_{\cdot,j+1}$ does contain initial points.
This implies $N(j+1)\leq\lfloor \lambda_{j+1}\rfloor +3 $.\\
(iii). By Lemma \ref{lemma:Nj}, $N(j)\geq \lfloor \lambda_j \rfloor +2$ and the vector $w_{\cdot,j+1}$
can not contain more than $\lfloor \lambda_{j+1} \rfloor+3$ non-zero entries.
\end{proof}

\subsubsection{Main Results: Analysis of {\sc (FFCRE)}}\label{sec:MainAnalysis}

We now present the main results concerning \ap{{\sc (FFCRE)}. We
first} show that the algorithm indeed delivers the correct fusion
frame, i.e., a fusion \ap{frame with} the prescribed fusion frame
operator. From this result, we deduce that in certain
cases a fusion frame can be turned into a tight fusion frame by
careful adding of new subsets (compare also \ap{with} Corollary
\ref{coro:waterfilling}).

\begin{theorem} \label{theo:givenFFOperator1}
Suppose the real values $\lambda_1\geq\cdots\geq\lambda_M$, $N \in \NN$, and $m \in \NN$
satisfy {\sc (Fac)} as well as the following conditions.
\bitem
\item[(i)] $\lambda_M \ge 2$.
\item[(ii)] If $j_0$ is the first integer in $\{1,\ldots,M\}$, for which $\lambda_{j_0}$ is not an integer, then
$\lfloor \lambda_{j_0} \rfloor \le N-3$.
\eitem
Then the fusion frame $\{\cW_i\}_{i=1}^N$ constructed by {\sc (FFCRE)} fulfills {\sc (FF1)} and {\sc (FF2)}.
\end{theorem}

\begin{proof}
If the set of vectors
\[
\{w_{i+km} : k = 0,\ldots,N-1\}
\]
is pairwise orthogonal for each $i = 1,\ldots,N$, then {\sc (FF1)} and {\sc (FF2)} follow automatically. Fix $i \in \{1,\ldots,N\}$. By construction, it \ap{is sufficient} to show that, for each $0 \le k \le N-2$, the vectors $w_{i+km}$ and $w_{i+(k+1)m}$ are disjointly supported. We distinguish between the following two cases:

{\it Case 1.} The vector $w_{i+km}$ is a unit vector, $e_n$, say. By (ii) and Lemma \ref{lemma:Nj},
$w_{\cdot,n}$ does not have more than $N$ non-zero elements. When defining the vector $w_{i+(k+1)m}$,
already $N-1$ vectors $w_\ell$ have been defined before its construction. Therefore this definition
takes place in a different step of the loop in Line 1). Hence $w_{i+(k+1)m,j} = 0$ for all $j=1,\ldots,n$.
This prove the claim in this case.

{\it Case 2.} The vector $w_{i+km}$ has two non-zero entries, namely an initial and a terminal point,
where the terminal point is at the $n$th position, say. Again, by (ii) and Lemma \ref{lemma:Nj},
$w_{\cdot,n}$ does not have more than $N$ non-zero elements. Hence, concluding as before,
$w_{i+(k+1)m,j} = 0$ for all $j=1,\ldots,n$. This prove the claim also in this case.
\end{proof}

Certainly, this theorem also holds in the special case of frames, i.e., $1$-dimensional
subspaces.

\begin{corollary} \label{coro:givenFFOperator1}
Suppose the real values $\lambda_1\geq\cdots\geq\lambda_M$ and $N \in \NN$
satisfy
\[
\sum_{j=1}^M \lambda_j = N
\]
as well as the following conditions.
\bitem
\item[(i)] $\lambda_M \ge 2$.
\item[(ii)] If $j_0$ is the first integer in $\{1,\ldots,M\}$, for which $\lambda_{j_0}$ is not an integer, then
$\lfloor \lambda_{j_0} \rfloor \le N-3$.
\eitem
Then the eigenvalues of the frame operator of the frame $\{w_k\}_{k=1}^N$ constructed by {\sc (FFCRE)}
are $\{\lambda_j\}_{j=1}^M$.
\end{corollary}

\begin{proof}
This result follows directly from Theorem \ref{theo:givenFFOperator1} by choosing $m=1$.
\end{proof}

Theorem \ref{theo:givenFFOperator1} is now applied to generate a tight fusion frame from a given fusion frame, satisfying some mild conditions.

\begin{theorem} \label{theo:givenFFOperator2}
Let $\{\cW_i\}_{i=1}^N$ be a fusion frame for $\bR^M$ with $\dim \cW_i=m<M$ for all
$i=1,\ldots ,N$, and let $S$ be the associated fusion frame operator with eigenvalues
$\lambda_1 \ge \ldots \ge \lambda_M$ and eigenvectors $\{e_j\}_{j=1}^M$.
Further, let $A$ be the smallest positive integer, which satisfies the following conditions.
\bitem
\item[(i)] $\lambda_1 +2 \le A$.
\item[(ii)] $AM = N_0m$ for some $N_0 \in \NN$.
\item[(iii)] $A \le \lambda_1 + N_0 - (N+3)$.
\eitem
Then there exists a fusion frame $\{\cV_i\}_{i=1}^{N_0-N}$ for $\bR^M$ with dim $\cV_i =m$ for all $i \in \{1,\ldots,N_0-N\}$
so that
\[
\{\cW_i\}_{i=1}^N \cup \{\cV_i\}_{i=1}^{N_0-N}
\]
is an $A$-tight fusion frame.
\end{theorem}

\begin{proof}
The first task is to check whether such a positive integer $A$ exists at all. We use
the ansatz $A = nm$ for some $n \in \NN$. This immediately satisfies (ii). Now choose
$n$ \ap{as} the smallest positive \ap{integer} still satisfying
\[
\lambda_1 + 2 \le A.
\]
Thus (i) and (ii) are fulfilled (and they will still be fulfilled for all larger $n\in \NN$.)
For inequality (iii), we require
\[
nm \le \lambda_1 + nM - (N+3),
\]
which we can reformulate as
\[
\frac{m}{M} \le \frac{\lambda_1}{Mn} + 1 - \frac{N+3}{Mn}.
\]
Since $\frac{m}{M} < 1$ by assumption, $n$ can be chosen large enough for this inequality
to be satisfied.

Next, we set
\[
\mu_j = A- \lambda_j \quad \mbox{for all } j=1,\ldots, M.
\]
In particular, we have $\mu_1 \ge \ldots \ge \mu_M$. We claim that the hypotheses of Theorem
\ref{theo:givenFFOperator1} are satisfied by the sequence $\{\mu_j\}_{j=1}^M$. For the proof,
we refer to the assumption of the present theorem as (i'), (ii'), and (iii').\\
(i). By (i'),
\[
\mu_1 = A-\lambda_1 \ge 2.
\]
Letting $N_1 = N_0-N$,
\[
\sum_{j=1}^{M} \mu_j
= \sum_{j=1}^M (A-\lambda_j)
= AM -\sum_{j=1}^M \lambda_j
= AM - Nm
= N_0m-Nm
= N_1m.
\]
(ii). By (iii'),
\[
\mu_1 = A-\lambda_1 \le (N_0-N)-3 = N_1-3.
\]

From Theorem \ref{theo:givenFFOperator1} it follows that there
exists a fusion frame $\{\cV_i\}_{i=1}^{N_1}$ for $\bR^M$ whose
fusion frame operator $S_1$, say, has eigenvectors $\{e_j\}_{j=1}^M$
and respective eigenvalues $\{\mu_j\}_{j=1}^M$. The fusion frame
operator for $\{\cW_i\}_{i=1}^N \cup \{\cV_i\}_{i=1}^{N_1}$ is
$S+S_1$, which then possesses as eigenvectors the sequence
$\{e_j\}_{j=1}^M$ with associated eigenvalues
\[
\lambda_j + \mu_j = \lambda_j +(A-\lambda_j) =A.
\]
Hence $\{\cW_i\}_{i=1}^N \cup \{\cV_i\}_{i=1}^{N_0-N}$ constitutes an $A$-tight fusion frame.
\end{proof}

The number of $m$-dimensional subspaces added in Theorem
\ref{theo:givenFFOperator2} to force a fusion frame to become tight
is in fact the \ap{smallest} number that can be added in general.
For this, let $\{\cW_i\}_{i=1}^N$ be a fusion frame for $\bR^M$ with
fusion frame operator $S$ having eigenvalues
$\{\lambda_j\}_{j=1}^M$. Suppose $\{\cV_i\}_{i=1}^{N_1}$ is any
family of $m$-dimensional subspaces with fusion frame operator
$S_1$, say, and so that the union of these two families is an
$A$-tight fusion frame for $\bR^M$. Thus
\[
S+S_1 = AI,
\]
which implies that the eigenvalues $\{\mu_j\}_{j=1}^M$ of $S_1$ satisfy
\[
\mu_j = A-\lambda_j \quad \mbox{for all } j=1,\ldots,M,
\]
and
\[
\sum_{j=1}^M \mu_j = \sum_{j=1}^M (A-\lambda_j) = AM-Nm = N_1m.
\]
In particular,
\[
AM = (N_1-N)m = N_0m.
\]
Thus, we have examples to show that -- in general -- fusion frames with the above properties of $S_1$ cannot be constructed unless the hypotheses of Theorem \ref{theo:givenFFOperator2} are satisfied. This shows that the smallest constant satisfying this theorem is in general the smallest number of subspaces we can add to obtain a tight fusion frame.

\subsection{Extensions and Related Problems}\label{subsec:extensions}

Finally, we would like to discuss several extensions and related problems.

{\em \ap{Weights.}} The handling of weights is \ap{particularly
delicate. When} turning a frame $\{f_i\}_{i=1}^N$ into the fusion
frame $\{(\spann\{f_i\},\|f_i\|)\}_{i=1}^N$ consisting of
1-dimensional subspaces and having the same (fusion) frame operator
as well as the same (fusion) frame bounds \cite[Prop. 2.14]{CKL08},
we notice that the subspaces are generated by the frame vectors and
the weights \ap{have to be chosen equal to} the norms of the frame vectors. Thus
choosing weights is in a sense comparable to choosing the lengths of
frame vectors. However \ap{the design is more} delicate due to
\ap{the necessary} compensation of the dimensions of the subspaces.
\ap{Generalizing,} for instance, Theorem \ref{theo:givenFFOperator1}
to weighted fusion frames \ap{requires careful} handling and \ap{a
thorough understanding of the interplay between subspace dimensions
and weights. This is currently under investigation.}

{\em \ap{Chordal Distances.}} It was shown in \cite{KPCL08} that
maximal resilience of fusion frames to noise and erasures \ap{is}
closely related to the chordal distances between \ap{pairs
of} subspaces forming \ap{the} fusion frame. Hence it would be
desirable to be able to control the set of chordal distances in
construction procedures for fusion frames. The results in Section
\ref{sec:generalconstruction} already allow this control. However,
for instance, for Theorem \ref{theo:givenFFOperator1} this control
is \ap{more difficult.}

\ap{{\em Sparsity.} A fusion frame (also in the special case of a
frame) with a desired fusion frame operator is often times designed
to provide an optimal tool for data processing with specific
performance metrics. The data processing ultimately needs to be
performed with a DSP board and therefore it is desired to reduce the
computational complexity (number of additions and multiplications
required) as much as possible. This motivates the design of fusion
frames, with fusion frame operators, which have sparsity properties. In other
words, it is desired for the vectors of a frame as well as the
generating vectors of the subspaces of a fusion frame to be sparse.
The sparsity allows} for fast vector-vector multiplications. In the
construction \ap{presented} in Section
\ref{sec:fusionframeoperator}, this principle is deployed by only
using sparse linear combinations of unit vectors. \ap{We conjecture
that this construction enjoys maximal sparsity. However, we do not
have a rigorous proof.} In general, generating frames and fusion
frames which allow for fast processing through additional inner
structural properties such as sparsity \ap{is becoming} more and
more indispensable.

{\em \ap{Equivalence Classes.}} Our results normally produce one
fusion frame satisfying a desired property. However, from a
scholarly point of view, it would be desirable to be able to
generate each such fusion frame in the sense of the whole
``equivalence class'' of fusion frames satisfying a special
property. This is beyond our reach at this point, since even the
following apparently simple problem is still unsolved: Construct one
Parseval frame in each equivalence class choosing unitary
equivalence as the relation.


\end{document}